\documentclass[12pt]{article}
\usepackage{mathrsfs}
\usepackage{amsfonts}
\usepackage{amsthm,amsmath,amssymb,anysize,longtable}
\usepackage{tikz}
\usepackage{float}
\usepackage{hyperref}
\usepackage{xcolor}
\usepackage{textcomp}
\usepackage{makecell}
\usepackage{longtable}
\usepackage{array}
\usepackage{caption}
\usepackage{float}
\usepackage{booktabs}
\usepackage{enumitem}

\newtheorem{theorem}{Theorem}
\newtheorem{lemma}[theorem]{Lemma}
\newtheorem{proposition}[theorem]{Proposition}

\newtheorem{example}[theorem]{Example}
\newtheorem{definition}[theorem]{Definition}
\newtheorem{corollary}[theorem]{Corollary}
\usepackage{xcolor}

\usepackage[numbers,sort&compress]{natbib}
\usepackage{amsmath}
\allowdisplaybreaks[3]
\marginsize{4.2cm}{4.2cm}{3.6cm}{3cm}
\numberwithin{equation}{section}

\author{\textsc{Baojin Zhang and Liming Tang\footnote{corresponding author}}\\
\small{School of Mathematical Sciences}\\
\small{Harbin Normal University}\\
\small{150025 Harbin, China}\\
\small{E-mail: limingtang@hrbnu.edu.cn}}

\date{ }
\date{ }
\usepackage[symbol]{footmisc}

\begin{document}

\thispagestyle{empty}

\noindent{\Large
The Solvabilizers and Solvable Graphs in Lie Superalgebras}
 \footnote{
The  work is supported by  the NSF of Hei Longjiang Province (No. LH2024A014)
}

	\bigskip
	
	 \bigskip

\begin{center}	
	{\bf
		
    Baojin Zhang\footnote{School of Mathematical Sciences, Harbin Normal University, 150025 Harbin, China; \ 2024400038@stu.hrbnu.edu.cn},
   \&
Liming Tang\footnote{School of Mathematical Sciences, Harbin Normal University, 150025 Harbin, China; \ limingtang@hrbnu.edu.cn}\footnote{corresponding author}}
\end{center}

 \begin{quotation}
{\small\noindent \textbf{Abstract}:
In this paper, we introduce the solvabilizer and the solvable graph for a Lie superalgebra and establish their basic properties.  Then we define a category which links Lie superalgebras to their solvable substructures.
Afterwards, we prove that the solvable graph is one of the isomorphic invariants of the Lie superalgebras.
Furthermore, we introduce the solvability measure,
which can reflect the degree of solvability of Lie superalgebras.

\medskip
 \vspace{0.05cm} \noindent{\textbf{Keywords}}:
Lie superalgebra; solvabilizer; graph theory.

\medskip

\vspace{0.05cm} \noindent \textbf{Mathematics Subject Classification
2020}: 17B05, 17B30, 18A05}
\end{quotation}
 \medskip

\section*{Introduction}
\ \ \
As a $\mathbb{Z}_{2}$-graded generalization of Lie algebras, the systematic theory of Lie superalgebras was first established by V. G. Kac \cite{KVG}.
Lie superalgebras play a vital role in the field of supersymmetry in physics, characterizing bosonic symmetries and fermionic symmetries.
Solvable Lie superalgebras are important special types of Lie superalgebras. They are defined by a sequence of subsuperalgebras that
eventually becomes trivial. These superalgebras are closely related to nilpotent structures and are widely studied in algebraic structure theory.

The study of local properties of algebraic structures has grown rapidly in recent decades. Such investigations are often
 closely related to graph theory which reveal deep connections between algebraic invariants and
topological properties. For example, In 1981, J. S. Williams \cite{JSW} introduced the prime graph of a finite group and
provided a classification of groups with disconnected prime graphs and some  key properties of the connected components of prime graphs.
In 2006, A. Abdollahi, S. Akbari, and H. R. Maimani \cite{AAM} introduced the non-commuting graph of groups
and established the connections between the graph properties and the group commutativity.
In 2025, D. Towers, I. Guti\'{e}rrez, and I. Fern\'{a}ndez \cite{TGF} introduced the solvableizer and the solvable graph in a Lie algebra
and developed computational frameworks based on GAP and SageMath.
More generally, such as prime power graphs of finite simple groups of Lie type \cite{KWMS}, non-nilpotent graph of groups \cite{AZ} and so on, please refer to \cite{DG,Zr,DN,BC,KVTC,GKLNS,THu,cmmo,BN,ACS,BMLAAA}.

Building on these ideas, we define the solvabilizer $\mathrm{sol}(L)$ and the solvable graph $\mathfrak{S}(L)$ for a
Lie superalgebra $L$ over a field $\mathbb{F}$ with characteristic $p\neq 2$. The vertex set is $L\setminus \operatorname{sol}(L)$, where
\begin{equation*}
\mathrm{sol}(L):= \{x\in L|\langle z,x\rangle\ is\ a\ solvable\ subsuperalgebra\ of\ L,\ for\ all\ z\in L\},
\end{equation*}
and the distinct vertices $x,y\in L\setminus \operatorname{sol}(L)$ are adjacent if the subsuperalgebra $\langle z,x\rangle$ is solvable.
For a Lie algebra, the center is always contained in the solvabilizer,
and every element is contained in its solvabilizer.
However, these properties do not generally hold for a Lie superalgebra.
Consequently, the solvabilizers of Lie superalgebras are more challenging to characterize.
In this paper, we consider Lie superalgebras over the field of characteristic $p\neq 2$.
For the case of $p=2$, the structure of the Lie superalgebras differs significantly. For details, readers are referred to \cite{BGLAEA,DLeites,AKALDLIS}.
The main results of this paper are as follows:

\begin{itemize}
\item The Extension Theorem and the properties for the solvabilizer under direct sums and homomorphisms are presented for a Lie superalgebra.


\item The solvable graph $\mathfrak{S}(L)$ and the solvability measure $\nu(L)$ are introduced for a Lie superalgebra $L$,
which quantify how close $L$ is to being solvable. We further prove that $\mathfrak{S}(L)$ and $\nu(L)$ are isomorphism invariants for a Lie superalgebra.

\item The functor $\mathcal{S}$ associates Lie superalgebras with their solvable substructures and the solvable graph category $\mathcal{G}$ associates Lie superalgebras with their solvable graphs are introduced.


\item Suppose that $L_{1}$ and $L_{2}$ are finite dimensional non-solvable Lie superalgebras over the field with $\text{char}\ \mathbb{F} \geq 3$.
Let $\varphi: L_1 \to L_2$ be a surjective homomorphism such that $\ker\varphi \subseteq \operatorname{sol}(L_1)$. Then
\[
\nu(L_2) \geq \nu(L_1),
\]
with equality if and only if $\varphi$ is an isomorphism of Lie superalgebras.
\end{itemize}

It is important to note that various results in this paper depend on the characteristic of the ground field $\mathbb{F}$.

\section{Basic}\label{section2}\ \ \

We briefly recall various facts associated with Lie superalgebras and Graph theory.
For more details, see \cite{RDi,MAAZAA}. Throughout this paper, $\mathbb{F}$ denotes a field;
$\operatorname{char} \mathbb{F}$ denotes the characteristic of $\mathbb{F}$; $\mathbb{Z}$ denotes the set of integers;
and $\mathbb{N}$ denotes the set of positive integers.
Let $\mathbb{Z}_{2} := \mathbb{Z}/2\mathbb{Z}$ be the ring of integers modulo 2. All Lie superalgebras considered are over a fixed field $\mathbb{F}$.
If $\operatorname{char} \mathbb{F}\neq 2$, a Lie superalgebra is a superalgebra $L = L_{\bar{0}} \oplus L_{\bar{1}}$
with an operation $[\cdot, \cdot]$ satisfying the following axioms:
\begin{enumerate}[label=(\arabic*)]
\item $[x,y]=-(-1)^{|x||y|}[y,x]$, \text{ for } $x,y\in L_{\bar{0}}\cup L_{\bar{1}}$,
\item $[x,[y,z]]=[[x,y],z]+(-1)^{|x||y|}[y,[x,z]]$, \text{ for } $x,y,z\in L_{\bar{0}}\cup L_{\bar{1}}$.
\end{enumerate}
Here $|x|=0$ if $x\in L_{\overline{0}}$ and $|x|=1$ if $x\in L_{\overline{1}}$. If $\operatorname{char} \mathbb{F}=3$,
The Jacobi identity for Lie superalgebras entails, additionally, that
$$[x,[x,x]]=0,\text{ for any }x\in L_{\overline{1}}.$$

A Lie superalgebra is said to be \emph{abelian} if $[x,y]=0$ for all $x,y\in L$.
The \emph{dimension} of Lie superalgebra $L$ is an ordered pair $(d_{0},d_{1})$,
where $d_{0}=\operatorname{dim}_{\mathbb{F}}L_{\overline{0}}$ and $d_{1}=\operatorname{dim}_{\mathbb{F}}L_{\overline{1}}$.
A $\mathbb{Z}_{2}$-graded subspace $I=I_{\overline{0}}\oplus I_{\overline{1}}$ of $L$ is called a \emph{graded ideal}
if $[x,y]\in I_{\overline{l+k}}$ for all $x\in I_{\overline{l}},y\in L_{\overline{k}}$,
where $\overline{l},\overline{k}\in \mathbb{Z}_{2}$ .
The intersections, sums, and Lie brackets of graded ideals are graded ideals.

The \emph{derived series} of $L$ is defined recursively by
\begin{equation*}
L^{(0)}=L,\ \ L^{(k+1)}=[L^{(k)},L^{(k)}],\ k\in \mathbb{Z}.
\end{equation*}

The \emph{descending central series} is given recursively by
\begin{equation*}
L^{1}=L,\ \ L^{k+1}=[L^{k},L],\ k\in \mathbb{Z}.
\end{equation*}

A Lie superalgebra $L$ is called \emph{solvable} if $L^{(r)}=0$ for some $r\in \mathbb{N}$, and \emph{nilpotent} if $L^{k}=0$ for some $k\in \mathbb{N}$.


A \emph{graph} $\Gamma = (V, E)$ consists of a non-empty vertex set $V$ and an edge set $E$,
where $E$ is a collection of unordered pairs of distinct elements from $V$. If $\{u, v\} \in E$,
then $u$ and $v$ are said to be \emph{adjacent}.
A \emph{subgraph} of $\Gamma= (V, E)$ is a graph $\Gamma' = (V', E')$ with $\emptyset \neq V' \subseteq V$ and $E' \subseteq E$.

A \emph{path} in $\Gamma= (V, E)$ is a sequence of distinct vertices $v_{1}, \dots, v_{k}$ such that $\{v_{i}, v_{i+1}\} \in E$ for all $1 \leq i < k$.
The graph $\Gamma= (V, E)$ is \emph{connected} if there exists a path between any two distinct vertices;
otherwise, it is \emph{disconnected}, and its maximal connected subgraph is called \emph{component};
the number of components is denoted by $\kappa(\Gamma)$. Two graphs $\Gamma_{1}= (V_{1}, E_{1})$ and $\Gamma_{2}= (V_{2}, E_{2})$ are said to be
\emph{isomorphic}, denoted $\Gamma_{1}\cong \Gamma_{2}$, if there exists a bijection $f:V_{1}\rightarrow V_{2}$ such that for every pair of distinct vertices
$u,v\in V_{1}$, $\{u,v\}\in E_{1}$ if and only if $\{f(u),f(v)\}\in E_{2}$. Such a bijection $f$ is called an \emph{isomorphism} between $\Gamma_{1}$ and $\Gamma_{2}$.

\section{The solvabilizer of a Lie superalgebra}\label{section5}
In this section, we establish some basic properties of the solvabilizer in a Lie superalgebra.
Throughout, the subalgebra refers to a subsuperalgebra,
and $\langle A\rangle$ denotes the subalgebra of $L$ generated by $A\in L$.

\begin{definition}
Suppose that $L$ is a finite dimensional Lie superalgebra.
\begin{enumerate}[label=(\arabic*)]
\item For non-empty subsets $A,B\subseteq L$, the solvabilizer of $B$ with respect to $A$ is defined as follows:
\begin{equation*}
\mathrm{sol}_{A}(B):= \{x\in A|\langle z,x\rangle\ is\ a\ solvable\ subalgebra\ of\ L,\ for\ all\ z\in B\}.
\end{equation*}

\item The solvabilizer of $L$ is defined by:
\begin{equation*}
\mathrm{sol}(L):= \{x\in L|\langle z,x\rangle\ is\ a\ solvable\ subalgebra\ of\ L,\ for\ all\ z\in L\}.
\end{equation*}
\end{enumerate}
Set $\mathrm{sol}_{\emptyset}(B):=\emptyset$, $\mathrm{sol}_{L}(z):=\mathrm{sol}_{L}(\{z\})$ and
$$\mathrm{sol}_{A}(\emptyset):=\{x\in A|\langle x\rangle\ is\ a\ solvable\ subalgebra\ of\ L\}.$$
\end{definition}

\noindent \textbf{Remark.}
 For a Lie superalgebra $L$ and $A\subseteq L$, $\mathrm{sol}_{A}(\emptyset)$ is not necessarily equal to $A$.
Since for an element $x\in A$, $\langle x\rangle$ is not necessarily a solvable subalgebra. However, $\mathrm{sol}_{A}(\emptyset)=A$ always holds in a Lie algebra.


\begin{example}
Let $L=L_{\overline{0}}\oplus L_{\overline{1}}$ be a Lie superalgebra with $\operatorname{char} \mathbb{F}=0$, where $L_{\overline{0}}=\operatorname{span}\{h\}$ and
$L_{\overline{1}}=\operatorname{span}\{x\}$. The Lie bracket is defined by
$$[h,x]=x,\ [x,x]=0.$$
It follows that $\operatorname{sol}_{L}(h) = L$ and $\operatorname{sol}(L) = L$.
\end{example}

Recall that the nilpotentizer for a Lie superalgebra \cite{AKALDLIS,HAS}.
Suppose $L$ is a finite dimensional Lie superalgebra in characteristic $p\neq 2$.
\begin{enumerate}[label=(\arabic*),font=\textnormal]
\item The nilpotentizer of $z$ in $L$ is defined by:
\begin{equation*}
\operatorname{nil}_{L}(z):= \{x\in L|\langle z,x\rangle\ is\ a\ nilpotent\ subalgebra\ of\ L\}.
\end{equation*}
\item The nilpotentizer of $L$ is defined by:
\begin{equation*}
\operatorname{nil}(L):= \{x\in L|\langle z,x\rangle\ is\ a\ nilpotent\ subsuperalgebra\ of\ L, \ for\ all\ z\in L\}.
\end{equation*}
\end{enumerate}

\begin{proposition}\label{}
Suppose that $L$ is a Lie superalgebra and $A,B,C$ are non-empty subsets of $L$.
Then the following properties hold:
\begin{enumerate}[label=(\arabic*)]
\item If $A\subseteq B$, then $\mathrm{sol}_{A}(C)\subseteq \mathrm{sol}_{B}(C)$ and $\mathrm{sol}_{C}(B)\subseteq \mathrm{sol}_{C}(A)$.

\item If $A\subseteq B$, then $\mathrm{sol}_{A}(C)=A\cap \mathrm{sol}_{B}(C)$.

\item $\mathrm{sol}_{C}(A\cup B)= \mathrm{sol}_{C}(A)\cap\mathrm{sol}_{C}(B)$.

\item $\mathrm{sol}_{A}(B)=\cap_{z\in B} \mathrm{sol}_{A}(z)$. In particular, $\mathrm{sol}(L)=\cap_{z\in L} \mathrm{sol}_{L}(z)$.

\item $\mathrm{nil}_{L}(z) \subseteq\mathrm \mathrm{sol}_{L}(z)$.

\item $\mathrm{nil}(L)\subseteq\mathrm{sol}(L)$.

\item If $L$ is solvable, then $\mathrm{sol}(L)=L$.

\item If $I$ is an ideal of $L$, then $\mathrm{sol}(L)\cap I\subseteq \mathrm{sol}(I)$.

\item If $J$ is a graded ideal of $L$, and $z\in L$, then $\frac{\mathrm{sol}_{L}(z)+J}{J}\subseteq \mathrm{sol}_{L/J}(z+J)$.

\item $\operatorname{sol}_{L}(z)$ is the union of all maximal solvable subalgebras of $L$ that contain $z$.
\end{enumerate}
\end{proposition}
\begin{proof}
Results (1)-(7) are straightforward.

(8) Suppose that $I$ is an ideal of $L$. For any $x\in \mathrm{sol}(L)\cap I$              
and any $z\in I$, the subalgebra $\langle z,x\rangle$ is solvable. It follows that $x\in \mathrm{sol}(I)$.
Therefore, $\mathrm{sol}(L)\cap I\subseteq \mathrm{sol}(I)$.

(9) Suppose that $J$ is an ideal of $L$. Let $x \in \operatorname{sol}_{L}(z)$. Then the subalgebra $\langle z, x \rangle$ is solvable.
Furthermore, we have $$\langle z + J, x + J \rangle = (\langle z, x \rangle + J)/J,$$
which is a homomorphic image of $\langle z, x \rangle$.
Since solvability is preserved under homomorphisms, it follows that $x + J \in \operatorname{sol}_{L/J}(z + J)$. This proves the inclusion
$$\frac{\mathrm{sol}_{L}(z)+J}{J}\subseteq \mathrm{sol}_{L/J}(z+J).$$

(10) Suppose that $x \in \operatorname{sol}_{L}(z)$.
Then the subalgebra $\langle z, x \rangle$ is solvable.
Consequently, it is contained in some maximal solvable subalgebra $M$ of $L$, and hence $x \in M$.
Conversely, if $x \in M$ for some maximal solvable subalgebra $M$ containing $z$, then $\langle z, x \rangle \subseteq M$ is solvable,
whence $x \in \operatorname{sol}_{L}(z)$.
\end{proof}







Next, we present the Extension Theorem for solvabilizer.
For a Lie superalgebra $L$ and $x,y\in L$, $\langle x, y \rangle_{L}$ denotes the subalgebra of $L$ generated by $x$ and $y$.

\begin{theorem}
Let $L$ be a finite dimensional Lie superalgebra and $\mathfrak{g}\subseteq L$ a subalgebra.
Then for any $x\in \mathfrak{g}$, we have
$$\operatorname{sol}_{L}(x)\cap \mathfrak{g}=\operatorname{sol}_{\mathfrak{g}}(x).$$
\end{theorem}
\begin{proof}
On the one hand, let $y\in \operatorname{sol}_{L}(x)\cap \mathfrak{g}$. Then $y\in L$ and the subalgebra $\langle x,y\rangle_{L}$ is solvable.
Since $x,y \in \mathfrak{g}$, we have $\langle x,y\rangle_{\mathfrak{g}}=\langle x,y\rangle_{L}$, which is therefore solvable.
Hence, $y\in \operatorname{sol}_{\mathfrak{g}}(x)$.
On the other hand, let $y \in \operatorname{sol}_{\mathfrak{g}}(x)$.
Then the subalgebra $\langle x, y \rangle_{\mathfrak{g}}$ is solvable.
As $\mathfrak{g} \subseteq L$, it follows that $\langle x, y \rangle_{\mathfrak{g}}$ is a solvable subalgebra of $L$,
and so $y \in \operatorname{sol}_{L}(x)$. Since $y$ also lies in $\mathfrak{g}$, we conclude that $y \in \operatorname{sol}_L(x) \cap \mathfrak{g}$.
\end{proof}

Next, we recall the definition of the direct sum for a Lie superalgebra.
Let $L = L_{\overline{0}} \oplus L_{\overline{1}}$ and $M = M_{\overline{0}} \oplus M_{\overline{1}}$ be Lie superalgebras over a field $\mathbb{F}$.
The direct sum of $L$ and $M$, denoted by $L \oplus M$, is the Lie superalgebra defined as follows:
\begin{enumerate}[label=(\arabic*)]
    \item  As a vector space, $L \oplus M$ is the direct sum of $L$ and $M$;
    \item $(L \oplus M)_{\alpha} = L_{\alpha} \oplus M_{\alpha}$ for $\alpha \in \mathbb{Z}_{2}$;
    \item For arbitrary $x_{1}, x_{2} \in L$ and $y_{1}, y_{2} \in M$,
    \[
    [x_1 + y_1, x_2 + y_2] = [x_1, x_2]_{L} + [y_1, y_2]_{M},
    \]
    where $[-, -]_{L}$ and $[-, -]_{M}$ denote the super brackets in $L$ and $ M$ respectively, with $[L, M] = \{0\}$.
\end{enumerate}

\begin{theorem}\label{sol direct sum}
Suppose that $L_{1}$ and $L_{2}$ are finite dimensional Lie superalgebras
and $L = L_{1} \oplus L_{2}$.
Then the following properties hold:
\begin{enumerate}[label=(\arabic*)]
\item The solvabilizer of $L$ satisfies:
$$\operatorname{sol}(L) = \operatorname{sol}(L_{1}) \oplus \operatorname{sol}(L_{2}).$$
\item For any element $x = (x_1, x_2) \in L$, the solvabilizer of $x$ satisfies:
$$\operatorname{sol}_L(x) = \operatorname{sol}_{L_{1}}(x_1) \oplus \operatorname{sol}_{L_{2}}(x_2).$$
\end{enumerate}
\end{theorem}
\begin{proof}
For any $x = (x_1, x_2)$, $y = (y_1, y_2)\in L$.
By the definition of a direct sum, we have
$$\langle x, y \rangle = \langle x_1, y_1 \rangle_{L_1} \oplus \langle x_2, y_2 \rangle_{L_2}.$$
A straightforward calculation shows that:

$$\langle x, y \rangle^{(k)} = \langle x_1, y_1 \rangle_{L_1}^{(k)} \oplus \langle x_2, y_2 \rangle_{L_2}^{(k)} \quad \text{for all } k \geq 2.$$
Consequently, $\langle x, y \rangle$ is solvable if and only if both $\langle x_1, y_1 \rangle_{L_1}$ and $\langle x_2, y_2 \rangle_{L_2}$ are solvable.
Since, we have the following chain of equivalences:
$$
\begin{aligned}
x \in \operatorname{sol}(L_1) \oplus \operatorname{sol}(L_2) &\Leftrightarrow  \langle x_{1}, y_{1} \rangle_{L_1} \text{ and } \langle x_{2}, y_{2} \rangle_{L_2} \text{ are solvable, for any } y = (y_1, y_2)\in L\\
&\Leftrightarrow  \langle x, y \rangle_{L_1 \oplus L_2} \text{ is solvable, for any } y = (y_1, y_2)\in L\\
&\Leftrightarrow  x \in \operatorname{sol}(L_1 \oplus L_2).
\end{aligned}
$$
This establishes that $\operatorname{sol}(L_1 \oplus L_2) = \operatorname{sol}(L_1) \oplus \operatorname{sol}(L_2)$.

(2) For any $x = (x_1, x_2)$ and $y = (y_1, y_2) \in L$, we have the following chain of equivalences:
$$
\begin{aligned}
y \in \operatorname{sol}_L(x) &\Leftrightarrow  \langle x, y \rangle \text{ is solvable} \\
&\Leftrightarrow  \langle x_1, y_1 \rangle_{L_1} \text{ and } \langle x_2, y_2 \rangle_{L_2} \text{ are solvable} \\
&\Leftrightarrow  y_1 \in \operatorname{sol}_{L_1}(x_1) \text{ and } y_2 \in \operatorname{sol}_{L_2}(x_2).
\end{aligned}
$$
Therefore,
$$
\operatorname{sol}_L(x) = \{ (y_1, y_2) \mid y_1 \in \operatorname{sol}_{L_1}(x_1), y_2 \in \operatorname{sol}_{L_2}(x_2) \} = \operatorname{sol}_{L_1}(x_1) \oplus \operatorname{sol}_{L_2}(x_2).
$$
\end{proof}

In general, if a Lie superalgebra $L$ decomposes as a direct sum of ideals $L=L_{1}\oplus L_{2}\oplus\cdots\oplus L_{n}$,
then for any element $x = (x_1,\cdots,x_n) \in L$,  the following identities hold:
\begin{equation*}
\mathrm{sol}(L)=\mathrm{sol}(L_{1})\oplus\cdots\oplus \mathrm{sol}(L_{n}),
\end{equation*}
and
\begin{equation*}
\operatorname{sol}_{L}(x)=\operatorname{sol}_{L_{1}}(x_{1}) \oplus\cdots\oplus \operatorname{sol}_{L_{n}}(x_{n}).
\end{equation*}

\begin{theorem}\label{sol12345}
Suppose that $L_{1},L_{2}$ are two finite dimensional Lie superalgebras, and let
$\varphi: L_{1}\longrightarrow L_{2}$ be a homomorphism.
Then the following conclusions hold:
\begin{enumerate}[label=(\arabic*)]
\item $\varphi(\operatorname{sol}(L_{1}))\subseteq \operatorname{sol}(\varphi(L_{1})).$

\item If $\operatorname{ker} \varphi\subseteq \operatorname{sol}(L_{1})$, then $\varphi(\operatorname{sol}(L_{1}))=\operatorname{sol}(\varphi(L_{1}))$.
\end{enumerate}

\end{theorem}
\begin{proof}
(1) Let $x \in \operatorname{sol}(L_{1})$. Then, for any $z \in L_1$, the subalgebra $\langle x, z \rangle$ is solvable.
Since $\varphi$ is a homomorphism, for any $l \in \varphi(L_{1})$, there exists $y \in L_{1}$ such that $l = \varphi(y)$.
The image $\langle \varphi(x), \varphi(y) \rangle = \varphi(\langle x, y \rangle)$ is a homomorphic image of a solvable algebra and is therefore solvable.
This implies that $\varphi(x) \in \operatorname{sol}(\varphi(L_{1}))$, and hence $\varphi(\operatorname{sol}(L_{1})) \subseteq \operatorname{sol}(\varphi(L_{1}))$.

(2) Let $y\in \operatorname{sol}(\varphi(L_{1}))$.
There exists $x\in L_{1}$ such that $\varphi(x)=y$. To prove the reverse inclusion, we must show that $x\in \operatorname{sol}(L_{1})$.
For any $z\in L_{1}$, we have
$\langle y, \varphi(z)\rangle=\langle \varphi(x), \varphi(z)\rangle=\varphi(\langle x,z\rangle)$ is solvable.
Let $K=\operatorname{ker}\ \varphi\cap \langle x,z\rangle$. Then we have the isomorphism
$\langle x,z\rangle/K\cong \varphi(\langle x,z\rangle)$,  which implies that $\langle x,z\rangle/K$ is solvable.
Furthermore, because of $K\subseteq \operatorname{ker}\ \varphi\subseteq \operatorname{sol}(L_{1})$, the ideal $K$ is solvable and thus
$\langle x,z\rangle$ is solvable.
Therefore, 
$\operatorname{sol}(\varphi(L_{1}))\subseteq \varphi(\operatorname{sol}(L_{1}))$.
In conclusion, $\varphi(\operatorname{sol}(L_{1}))=\operatorname{sol}(\varphi(L_{1}))$.
\end{proof}

\begin{corollary}\label{solco}
Suppose that $L=L_{1}\oplus\cdots\oplus L_{n}$ and $M=M_{1}\oplus\cdots\oplus M_{n}$ are two direct sums of finite dimensional Lie superalgebras,
and let
$\varphi=\varphi_{1}+ \cdots+ \varphi_{n}: L\longrightarrow M$ be a homomorphism,
where $\varphi_{i}: L_{i}\longrightarrow M_{i}(1\leq i \leq n)$ is a homomorphism and $\varphi_{i}(L_{j})=0$ for $i \neq j$.
If $ker\ \varphi\subseteq \operatorname{nil}(L)$,
then $$\operatorname{sol}(\varphi(L))=\varphi_{1}(\operatorname{sol}(L_{1}))\oplus\cdots \oplus \varphi_{n}(\operatorname{sol}(L_{n})).$$
\end{corollary}
\begin{proof}
By Theorems \ref{sol direct sum} and \ref{sol12345},
we have $$\operatorname{sol}(\varphi(L))=\varphi(\operatorname{sol}(L))=\varphi(\mathrm{sol}(L_{1})\oplus\cdots\oplus \mathrm{sol}(L_{n}))=
\varphi_{1}(\operatorname{sol}(L_{1}))\oplus\cdots\oplus \varphi_{n}(\operatorname{sol}(L_{n}).$$
\end{proof}

\begin{proposition}
Let $\mathcal{L}$ denote the category defined as follows:
\begin{itemize}
    \item \textbf{Objects}: Finite dimensional Lie superalgebras;
    \item \textbf{Morphisms}: Surjective homomorphisms of Lie superalgebras.
\end{itemize}

Define a map $\mathcal{S}:\mathcal{L}\to\mathcal{L}$ by specifying its action:
\begin{itemize}
    \item \textbf{On objects}: For any finite dimensional Lie superalgebra $L\in\mathcal{L}$, define $\mathcal{S}(L)=\mathrm{sol}(L)$ ;
    \item \textbf{On morphisms}: For any morphism $f:L\to M$ in $\mathcal{L}$, let $\mathcal{S}(f)$ be the restriction of $f$ to $\mathrm{sol}(L)$, that is,
    $\mathcal{S}(f) = f|_{\mathrm{sol}(L)}:\mathrm{sol}(L)\to\mathrm{sol}(M)$.
\end{itemize}

Then $\mathcal{S}$ is a functor from $\mathcal{L}$ to itself.
\begin{proof}
By Proposition 22 in \cite{ZT}, the category $L$ is well-defined. By Theorem \ref{sol12345}(1), for any homomorphism $f: L_{1} \to L_{2}$, $f(\text{sol}(L_{1})) \subseteq \text{sol}(L_{2})$.
This ensures that the restriction $\mathcal{S}(f)$ is well-defined.
To verify that $\mathcal{S}$ is a functor, we need to check two key functorial properties: the preservation of identity morphisms and the preservation of composition of morphisms.

(1) Preservation of Identity Morphisms.

Let $L$ be an object in $\mathcal{L}$, and let $\operatorname{id}_{L}: L \to L$ be its identity morphism.
By the definition of the functor $\mathcal{S}$, we have $ \mathcal{S}(L) = \operatorname{sol}(L)$ on objects,
and $\mathcal{S}(\operatorname{id}_L) = \operatorname{id}_L|_{\operatorname{sol}(L)}$ on morphisms.
The following we show that this restricted map is the identity morphism on $\mathcal{S}(L)$.
For any $x \in \operatorname{sol}(L)$, it follows directly that
$$\mathcal{S}(\operatorname{id}_L)(x) = \operatorname{id}_L|_{\operatorname{sol}(L)}(x) = \operatorname{id}_L(x) = x.$$

Therefore, $\mathcal{S}(\operatorname{id}_L) = \operatorname{id}_{\operatorname{sol}(L)} = \operatorname{id}_{\mathcal{S}(L)}$,
which proves that $\mathcal{S}$ preserves identity morphisms.

(2) Preservation of Composition of Morphisms.

Let $f: L \to M $ and $g: M \to N $ be morphisms in $\mathcal{L}$. We show that $\mathcal{S}(g \circ f) = \mathcal{S}(g) \circ \mathcal{S}(f)$.
For any $x \in \operatorname{sol}(L)$, we compute both sides:
\begin{equation*}
\mathcal{S}(g \circ f)(x)=(g \circ f)|_{\text{sol}(L)}(x)=(g \circ f)(x)=g(f(x)).
\end{equation*}
On the other hand,
\begin{equation*}
\mathcal{S}(g) \circ \mathcal{S}(f)(x)=g|_{\text{sol}(M)}(f|_{\text{sol}(L)}(x))=g(f(x)).
\end{equation*}
We conclude that $\mathcal{S}(g \circ f) = \mathcal{S}(g) \circ \mathcal{S}(f)$.
Since $\mathcal{S}$ preserves both identity morphisms and composition of morphisms, it is a functor from $\mathcal{L}$ to $\mathcal{L}$.
\end{proof}
\end{proposition}

\begin{theorem}
Suppose that $A,\ B,\ C$ are Lie superalgebras and let $$0 \to A \xrightarrow{\alpha} B \xrightarrow{\beta} C \to 0$$ be a short exact sequence.
Then there exist natural homomorphisms:
\begin{equation*}
\alpha': \mathrm{sol}(A) \to B, \quad \beta': \mathrm{sol}(B) \to \mathrm{sol}(C)
\end{equation*}
such that $\ker \beta' \subseteq \alpha'(\mathrm{sol}(A))$.
\end{theorem}
\begin{proof}
Let $\alpha' = \alpha|_{\operatorname{sol}(A)}$ and $\beta' = \beta|_{\operatorname{sol}(B)}$.
First we show that $\beta'(\mathrm{sol}(B))\subseteq \operatorname{sol}(C)$.
Take $x\in\operatorname{sol}(B)$ and any $c\in C$.
Since $\beta$ is surjective, choose a $b\in B$ such that $\beta(b)=c$.
Since $\langle x,b\rangle$ is solvable and $\beta$ is a homomorphism,
$\langle\beta(x),c\rangle=\beta(\langle x,b\rangle)$ is also solvable.
Then $\beta(x)\in\operatorname{sol}(C)$.

Now let $x\in\ker\beta'\subseteq\operatorname{sol}(B)$.
Then $\beta(x)=0$, so $x\in\ker\beta=\operatorname{im}\alpha$.
Write $x=\alpha(a)$ for some $a\in A$. We prove that $a\in\operatorname{sol}(A)$.
For any $a'\in A$,
\[
\alpha(\langle a,a'\rangle)=\langle\alpha(a),\alpha(a')\rangle=\langle x,\alpha(a')\rangle,
\]
which is solvable because $x\in\operatorname{sol}(B)$.
Since $\alpha$ is an injective homomorphism, then $\langle a,a'\rangle$ must be solvable.
Thus $a\in\operatorname{sol}(A)$ and consequently $x=\alpha(a)\in\alpha'(\operatorname{sol}(A))$.
\end{proof}

\begin{theorem}
Suppose that $f: L \to M$ and $g: N \to M$  are morphisms in the category $\mathcal{L}$.
Define
$$P =L \times_M N=\{(x, y) \in L \oplus N \mid f(x) = g(y)\},$$
and let $p_L: P \to L$ and $p_N: P \to N$ be the canonical projections.
Then the following statements hold:
\begin{enumerate}[label=(\arabic*)]
\item The pair $(p_L,p_N)$ is the pullback of $f$ and $g$.

\item $\mathcal{S}(P) =\mathcal{S}(L) \times_{\mathcal{S}(M)} \mathcal{S}(N)$,
where
$$\mathcal{S}(L) \times_{\mathcal{S}(M)} \mathcal{S}(N)=\{(x, y) \in \mathcal{S}(L) \oplus \mathcal{S}(N) \mid \mathcal{S}(f)(x) = \mathcal{S}(g)(y)\}.$$
\end{enumerate}
\end{theorem}

\begin{proof}
(1) It is clear that $P$ is a subalgebra of $L \oplus N$.
Take any $x \in L$. Since $g$ is surjective, there exists some $y \in N$ such that $g(y) = f(x)$.
It follows that $p_L$ is surjective. Similarly, $p_N$ is surjective.

Let $K$ be an object in $\mathcal{L}$, and let $u: K \to L$ and $v: K \to N$ be surjective homomorphisms satisfying $f \circ u = g \circ v$.
Define a map $h:K\rightarrow P$ by $h(k) = (u(k), v(k))$ for all $k \in K$.
The condition $f \circ u = g \circ v$ ensures that $f(u(k)) = g(v(k))$, so $h(k) \in P$.
By construction, $h$ is a morphism in $\mathcal{L}$.
For the uniqueness of $h$, suppose $h': K \to P$ is another morphism such that $p_L \circ h' = u$ and $p_N \circ h' = v$.
Then for any $k \in K$,
$$h'(k) = (p_L(h'(k)), p_N(h'(k))) = (u(k), v(k)) = h(k).$$
Therefore, $h'=h$.

(2) Suppose that $(x, y) \in \mathcal{S}(P)$. Then for any $(a, b) \in P$, the subalgebra $\langle (x, y), (a, b) \rangle$ is solvable.
By Theorem \ref{sol direct sum}, this is equivalent to $\langle x, a \rangle$ and $\langle y, b \rangle$ being solvable.
Take any $a \in L$. Since $g$ is surjective, there exists $b \in N$ such that $g(b) = f(a)$. Hence $(a, b) \in P$ and $\langle x, a \rangle$ is solvable,
which proves $x \in \mathcal{S}(L)$. Similarly, $y \in \mathcal{S}(N)$. Moreover, the condition $f(x) = g(y)$ holds,
it follows that
$$\mathcal{S}(P) \subseteq\mathcal{S}(L) \times_{\mathcal{S}(M)} \mathcal{S}(N).$$
Conversely, suppose that $(x, y) \in \mathcal{S}(L) \times_{\mathcal{S}(M)} \mathcal{S}(N)$.
Then for any $(a, b) \in P$, the subalgebras $\langle x, a \rangle$ and $\langle y, b \rangle$ are solvable.
Applying Theorem \ref{sol direct sum} again, we conclude that $\langle (x, y), (a, b) \rangle$ is solvable, which means $(x, y) \in \mathcal{S}(P)$.
Thus, $$\mathcal{S}(L) \times_{\mathcal{S}(M)} \mathcal{S}(N)\subseteq\mathcal{S}(P).$$
Therefore, $$\mathcal{S}(P) = \mathcal{S}(L) \times_{\mathcal{S}(M)} \mathcal{S}(N).$$
\end{proof}

\section{The Solvable Graph of Lie Superalgebras}
\begin{definition}\label{graphsol}
Let $L$ be a finite dimensional non-solvable Lie superalgebra.
The solvable graph of $L$, denoted by $\mathfrak{S}(L)=(V,E)$,  is defined as follows:
\begin{enumerate}[label=(\arabic*)]
\item  The vertex set $V=L\setminus (\operatorname{sol}(L)\cup \{0\})$.

\item For distinct vertices $x,y\in V$, $\{x,y\}\in E$ if and only if the subalgebra $\langle x,y\rangle$ is solvable.
\end{enumerate}
\end{definition}

\begin{definition}
Let that $L$ be a finite dimensional non-solvable Lie superalgebra.
The non-solvable graph of $L$, denoted by $\mathfrak{S}(L)^{c}=(V,E)$,  is defined as follows:
\begin{enumerate}[label=(\arabic*)]
\item The vertex set $V=L\setminus (\operatorname{sol}(L)\cup \{0\})$.

\item For distinct vertices $x,y\in V$, $\{x,y\}\in E$ if and only if the subalgebra $\langle x,y\rangle$ is non-solvable.
\end{enumerate}
\end{definition}

\begin{definition}
Let $L$ be a finite dimensional non-solvable Lie superalgebra with char $\mathbb{F}\geq3$ and
$\mathfrak{G}(L)=(V,E)$ be its solvable graph.
We define the solvability measure $\nu(L)$ of $L$ as:
$$\nu(L) = 1 - \frac{2 |E|}{|V|(|V| - 1)},$$
where $|E|$ and $|V|$ denote the number of elements in $E$ and $V$, respectively.
\end{definition}
Since $|V|\geq 2$, the definition is well-defined.

\begin{example}
Let $L=L_{\overline{0}}\oplus L_{\overline{1}}$ be a Lie superalgebra, where $L_{\overline{0}}={\rm{span}}\{h\}$ and
$L_{\overline{1}}={\rm{span}}\{x,y\}$. Suppose that the Lie bracket is defined by the following relations:
$$[h,x]=x,\ [h,y]=-y,\ [x,y]=h,\ [x,x]=[y,y]=0.$$
 Its  solvable graph is shown as follows:
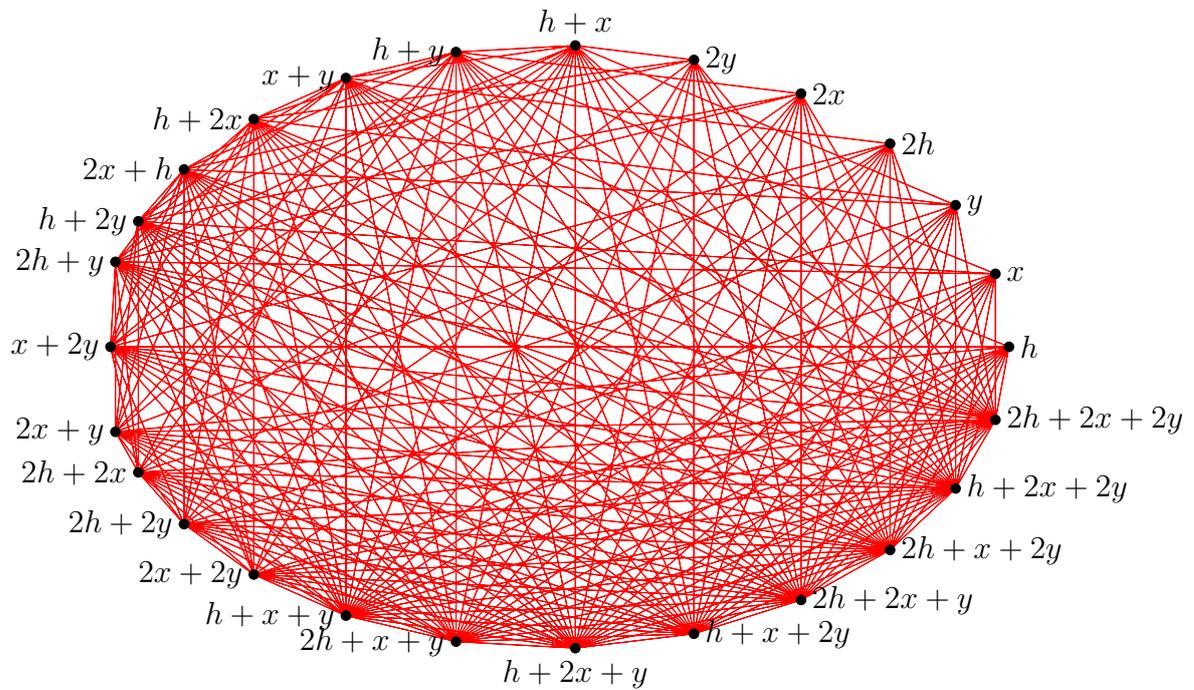
\begin{longtable}{m{16cm}<{\centering}}
        \toprule
         solvable graph\\ \hline
        \endhead
        \bottomrule
        \endfoot
        \bottomrule
        \endlastfoot
       \\
       \begin{tikzpicture}
\fill(5.000000*1.2,0.000000*0.8)circle(2pt)coordinate(dot);
\node[right] at (dot){$h$};
\fill(4.851480*1.2,1.209610*0.8)circle(2pt)coordinate(dot);
\node[right] at (dot){$x$};
\fill(4.414740*1.2,2.347360*0.8)circle(2pt)coordinate(dot);
\node[right] at (dot){$y$};
\fill(3.695964*1.2,3.366724*0.8)circle(2pt)coordinate(dot);
\node[right] at (dot){$2h$};
\fill(2.720104*1.2,4.195360*0.8)circle(2pt)coordinate(dot);
\node[right] at (dot){$2x$};
\fill(1.545084*1.2,4.755284*0.8)circle(2pt)coordinate(dot);
\node[right] at (dot){$2y$};
\fill(1.545084*1.2,-4.755284*0.8)circle(2pt)coordinate(dot);
\node[right] at (dot){$h+x+2y$};
\fill(2.720104*1.2,-4.195360*0.8)circle(2pt)coordinate(dot);
\node[right] at (dot){$2h+2x+y$};
\fill(3.695964*1.2,-3.366724*0.8)circle(2pt)coordinate(dot);
\node[right] at (dot){$2h+x+2y$};
\fill(4.414740*1.2,-2.347360*0.8)circle(2pt)coordinate(dot);
\node[right] at (dot){$h+2x+2y$};
\fill(4.851480*1.2,-1.209610*0.8)circle(2pt)coordinate(dot);
\node[right] at (dot){$2h+2x+2y$};

\fill(-1.062784*1.2,4.885640*0.8)circle(2pt)coordinate(dot);
\node[left] at (dot){$h+y$};
\fill(-2.269950*1.2,4.455034*0.8)circle(2pt)coordinate(dot);
\node[left] at (dot){$x+y$};
\fill(-3.280294*1.2,3.773550*0.8)circle(2pt)coordinate(dot);
\node[left] at (dot){$h+2x$};
\fill(-4.045084*1.2,2.938924*0.8)circle(2pt)coordinate(dot);
\node[left] at (dot){$2x+h$};
\fill(-4.546740*1.2,2.080734*0.8)circle(2pt)coordinate(dot);
\node[left] at (dot){$h+2y$};
\fill(-4.797464*1.2,1.408664*0.8)circle(2pt)coordinate(dot);
\node[left] at (dot){$2h+y$};
\fill(-4.851480*1.2,0.000000*0.8)circle(2pt)coordinate(dot);
\node[left] at (dot){$x+2y$};
\fill(-4.797464*1.2,-1.408664*0.8)circle(2pt)coordinate(dot);
\node[left] at (dot){$2x+y$};
\fill(-4.546740*1.2,-2.080734*0.8)circle(2pt)coordinate(dot);
\node[left] at (dot){$2h+2x$};
\fill(-4.045084*1.2,-2.938924*0.8)circle(2pt)coordinate(dot);
\node[left] at (dot){$2h+2y$};
\fill(-3.280294*1.2,-3.773550*0.8)circle(2pt)coordinate(dot);
\node[left] at (dot){$2x+2y$};
\fill(-2.269950*1.2,-4.455034*0.8)circle(2pt)coordinate(dot);
\node[left] at (dot){$h+x+y$};
\fill(-1.062784*1.2,-4.885640*0.8)circle(2pt)coordinate(dot);
\node[left] at (dot){$2h+x+y$};

\fill(0.245144*1.2,4.993920*0.8)circle(2pt)coordinate(dot);
\node[above] at (dot){$h+x$};
\fill(0.245144*1.2,-4.993920*0.8)circle(2pt)coordinate(dot);
\node[below] at (dot){$h+2x+y$};

\coordinate (v0) at (5.000000*1.2,0.000000*0.8);       
\coordinate (v1) at (4.851480*1.2,1.209610*0.8);       
\coordinate (v2) at (4.414740*1.2,2.347360*0.8);       
\coordinate (v3) at (3.695964*1.2,3.366724*0.8);       
\coordinate (v4) at (2.720104*1.2,4.195360*0.8);       
\coordinate (v5) at (1.545084*1.2,4.755284*0.8);       
\coordinate (v6) at (0.245144*1.2,4.993920*0.8);       
\coordinate (v7) at (-1.062784*1.2,4.885640*0.8);      
\coordinate (v8) at (-2.269950*1.2,4.455034*0.8);      
\coordinate (v9) at (-3.280294*1.2,3.773550*0.8);      
\coordinate (v10) at (-4.045084*1.2,2.938924*0.8);     
\coordinate (v11) at (-4.546740*1.2,2.080734*0.8);     
\coordinate (v12) at (-4.797464*1.2,1.408664*0.8);     
\coordinate (v13) at (-4.851480*1.2,0.000000*0.8);     
\coordinate (v14) at (-4.797464*1.2,-1.408664*0.8);    
\coordinate (v15) at (-4.546740*1.2,-2.080734*0.8);    
\coordinate (v16) at (-4.045084*1.2,-2.938924*0.8);    
\coordinate (v17) at (-3.280294*1.2,-3.773550*0.8);    
\coordinate (v18) at (-2.269950*1.2,-4.455034*0.8);    
\coordinate (v19) at (-1.062784*1.2,-4.885640*0.8);    
\coordinate (v20) at (0.245144*1.2,-4.993920*0.8);     
\coordinate (v21) at (1.545084*1.2,-4.755284*0.8);     
\coordinate (v22) at (2.720104*1.2,-4.195360*0.8);     
\coordinate (v23) at (3.695964*1.2,-3.366724*0.8);     
\coordinate (v24) at (4.414740*1.2,-2.347360*0.8);     
\coordinate (v25) at (4.851480*1.2,-1.209610*0.8);     

\draw[blue, thin] (v1) -- (v2) -- (v0) -- (v4) -- (v5) -- (v3) -- (v1);
\draw[blue, thin] (v1) -- (v5) -- (v0) -- (v2) -- (v4) -- (v3) -- (v2);
\draw[blue, thin] (v1) -- (v3) -- (v5) -- (v1); 

\draw[blue, thin] (v1) -- (v8) -- (v2);
\draw[blue, thin] (v1) -- (v13) -- (v5);
\draw[blue, thin] (v4) -- (v14) -- (v5);
\draw[blue, thin] (v4) -- (v17) -- (v5);
\draw[blue, thin] (v2) -- (v13) -- (v4);
\draw[blue, thin] (v2) -- (v17) -- (v4);

\draw[blue, thin] (v1) -- (v6) -- (v0);
\draw[blue, thin] (v1) -- (v9) -- (v0);
\draw[blue, thin] (v4) -- (v10) -- (v0);
\draw[blue, thin] (v4) -- (v15) -- (v3);
\draw[blue, thin] (v0) -- (v15) -- (v4);
\draw[blue, thin] (v3) -- (v6) -- (v1);

\draw[blue, thin] (v2) -- (v7) -- (v0);
\draw[blue, thin] (v2) -- (v11) -- (v0);
\draw[blue, thin] (v5) -- (v12) -- (v3);
\draw[blue, thin] (v5) -- (v16) -- (v3);
\draw[blue, thin] (v0) -- (v16) -- (v5);
\draw[blue, thin] (v3) -- (v7) -- (v2);

\draw[blue, thin] (v8) -- (v13) -- (v14) -- (v17) -- (v8); 
\draw[blue, thin] (v8) -- (v14); 
\draw[blue, thin] (v13) -- (v17); 

\draw[blue, thin] (v6) -- (v9) -- (v10) -- (v15) -- (v6); 
\draw[blue, thin] (v6) -- (v10); 
\draw[blue, thin] (v9) -- (v15); 
\draw[blue, thin] (v6) -- (v15); 
\draw[blue, thin] (v9) -- (v10); 

\draw[blue, thin] (v7) -- (v11) -- (v12) -- (v16) -- (v7); 
\draw[blue, thin] (v7) -- (v12); 
\draw[blue, thin] (v11) -- (v16); 
\draw[blue, thin] (v7) -- (v16); 
\draw[blue, thin] (v11) -- (v12); 

\draw[blue, thin] (v1) -- (v0);    
\draw[blue, thin] (v0) -- (v3);    
\draw[blue, thin] (v4) -- (v6);    
\draw[blue, thin] (v5) -- (v7);    
\draw[blue, thin] (v4) -- (v8);    
\draw[blue, thin] (v5) -- (v8);    
\draw[blue, thin] (v4) -- (v9);    
\draw[blue, thin] (v4) -- (v10);   
\draw[blue, thin] (v5) -- (v11);   
\draw[blue, thin] (v5) -- (v12);   

\draw[blue, thin] (v0) -- (v12);   
\draw[blue, thin] (v1) -- (v4);
\draw[blue, thin] (v1) -- (v10);
\draw[blue, thin] (v1) -- (v14);
\draw[blue, thin] (v1) -- (v15);
\draw[blue, thin] (v1) -- (v17);
\draw[blue, thin] (v2) -- (v5);
\draw[blue, thin] (v2) -- (v12);
\draw[blue, thin] (v2) -- (v14);
\draw[blue, thin] (v2) -- (v16);

\draw[blue, thin] (v3) -- (v9);
\draw[blue, thin] (v3) -- (v10);

\fill(v0)circle(2pt)coordinate(dot);\node[right] at (dot){$h$};
\fill(v1)circle(2pt)coordinate(dot);\node[right] at (dot){$x$};
\fill(v2)circle(2pt)coordinate(dot);\node[right] at (dot){$y$};
\fill(v3)circle(2pt)coordinate(dot);\node[right] at (dot){$2h$};
\fill(v4)circle(2pt)coordinate(dot);\node[right] at (dot){$2x$};
\fill(v5)circle(2pt)coordinate(dot);\node[right] at (dot){$2y$};
\fill(v6)circle(2pt)coordinate(dot);\node[above] at (dot){$h+x$};
\fill(v7)circle(2pt)coordinate(dot);\node[left] at (dot){$h+y$};
\fill(v8)circle(2pt)coordinate(dot);\node[left] at (dot){$x+y$};
\fill(v9)circle(2pt)coordinate(dot);\node[left] at (dot){$h+2x$};
\fill(v10)circle(2pt)coordinate(dot);\node[left] at (dot){$2x+h$};
\fill(v11)circle(2pt)coordinate(dot);\node[left] at (dot){$h+2y$};
\fill(v12)circle(2pt)coordinate(dot);\node[left] at (dot){$2h+y$};
\fill(v13)circle(2pt)coordinate(dot);\node[left] at (dot){$x+2y$};
\fill(v14)circle(2pt)coordinate(dot);\node[left] at (dot){$2x+y$};
\fill(v15)circle(2pt)coordinate(dot);\node[left] at (dot){$2h+2x$};
\fill(v16)circle(2pt)coordinate(dot);\node[left] at (dot){$2h+2y$};
\fill(v17)circle(2pt)coordinate(dot);\node[left] at (dot){$2x+2y$};
\fill(v18)circle(2pt)coordinate(dot);\node[left] at (dot){$h+x+y$};
\fill(v19)circle(2pt)coordinate(dot);\node[left] at (dot){$2h+x+y$};
\fill(v20)circle(2pt)coordinate(dot);\node[below] at (dot){$h+2x+y$};
\fill(v21)circle(2pt)coordinate(dot);\node[right] at (dot){$h+x+2y$};
\fill(v22)circle(2pt)coordinate(dot);\node[right] at (dot){$2h+2x+y$};
\fill(v23)circle(2pt)coordinate(dot);\node[right] at (dot){$2h+x+2y$};
\fill(v24)circle(2pt)coordinate(dot);\node[right] at (dot){$h+2x+2y$};
\fill(v25)circle(2pt)coordinate(dot);\node[right] at (dot){$2h+2x+2y$};
\end{tikzpicture}
\end{longtable}

\begin{longtable}{m{16cm}<{\centering}}
        \toprule
         non-solvable graph\\ \hline
        \endhead
        \bottomrule
        \endfoot
        \bottomrule
        \endlastfoot
       \\
       \begin{tikzpicture}
\coordinate (v0) at (5.000000*1.2,0.000000*0.8);       
\coordinate (v1) at (4.851480*1.2,1.209610*0.8);   
\coordinate (v2) at (4.414740*1.2,2.347360*0.8);   
\coordinate (v3) at (3.695964*1.2,3.366724*0.8);   
\coordinate (v4) at (2.720104*1.2,4.195360*0.8);   
\coordinate (v5) at (1.545084*1.2,4.755284*0.8);   
\coordinate (v6) at (0.245144*1.2,4.993920*0.8);   
\coordinate (v7) at (-1.062784*1.2,4.885640*0.8);  
\coordinate (v8) at (-2.269950*1.2,4.455034*0.8);  
\coordinate (v9) at (-3.280294*1.2,3.773550*0.8);  
\coordinate (v10) at (-4.045084*1.2,2.938924*0.8); 
\coordinate (v11) at (-4.546740*1.2,2.080734*0.8); 
\coordinate (v12) at (-4.797464*1.2,1.408664*0.8); 
\coordinate (v13) at (-4.851480*1.2,0.000000*0.8); 
\coordinate (v14) at (-4.797464*1.2,-1.408664*0.8);
\coordinate (v15) at (-4.546740*1.2,-2.080734*0.8);
\coordinate (v16) at (-4.045084*1.2,-2.938924*0.8);
\coordinate (v17) at (-3.280294*1.2,-3.773550*0.8);
\coordinate (v18) at (-2.269950*1.2,-4.455034*0.8);
\coordinate (v19) at (-1.062784*1.2,-4.885640*0.8);
\coordinate (v20) at (0.245144*1.2,-4.993920*0.8);
\coordinate (v21) at (1.545084*1.2,-4.755284*0.8);
\coordinate (v22) at (2.720104*1.2,-4.195360*0.8);
\coordinate (v23) at (3.695964*1.2,-3.366724*0.8);
\coordinate (v24) at (4.414740*1.2,-2.347360*0.8);
\coordinate (v25) at (4.851480*1.2,-1.209610*0.8);

\draw[red, thin] 
(v0)--(v8) (v0)--(v13) (v0)--(v14) (v0)--(v17) (v0)--(v18) (v0)--(v19) (v0)--(v20) (v0)--(v21) (v0)--(v22) (v0)--(v23) (v0)--(v24) (v0)--(v25)
(v1)--(v7) (v1)--(v11) (v1)--(v12) (v1)--(v16) (v1)--(v18) (v1)--(v19) (v1)--(v20) (v1)--(v21) (v1)--(v22) (v1)--(v23) (v1)--(v24) (v1)--(v25)
(v2)--(v6) (v2)--(v9) (v2)--(v10) (v2)--(v15) (v2)--(v18) (v2)--(v19) (v2)--(v20) (v2)--(v21) (v2)--(v22) (v2)--(v23) (v2)--(v24) (v2)--(v25)
(v3)--(v8) (v3)--(v13) (v3)--(v14) (v3)--(v17) (v3)--(v18) (v3)--(v19) (v3)--(v20) (v3)--(v21) (v3)--(v22) (v3)--(v23) (v3)--(v24) (v3)--(v25)
(v4)--(v7) (v4)--(v11) (v4)--(v12) (v4)--(v16) (v4)--(v18) (v4)--(v19) (v4)--(v20) (v4)--(v21) (v4)--(v22) (v4)--(v23) (v4)--(v24) (v4)--(v25)
(v5)--(v6) (v5)--(v9) (v5)--(v10) (v5)--(v15) (v5)--(v18) (v5)--(v19) (v5)--(v20) (v5)--(v21) (v5)--(v22) (v5)--(v23) (v5)--(v24) (v5)--(v25)
(v6)--(v2) (v6)--(v5) (v6)--(v7) (v6)--(v8) (v6)--(v11) (v6)--(v12) (v6)--(v13) (v6)--(v14) (v6)--(v16) (v6)--(v17) (v6)--(v18) (v6)--(v19) (v6)--(v20) (v6)--(v21) (v6)--(v22) (v6)--(v23) (v6)--(v24) (v6)--(v25)
(v7)--(v1) (v7)--(v4) (v7)--(v6)  (v7)--(v8) (v7)--(v9) (v7)--(v10) (v7)--(v13) (v7)--(v14) (v7)--(v15) (v7)--(v17) (v7)--(v18) (v7)--(v19) (v7)--(v20) (v7)--(v21) (v7)--(v22) (v7)--(v23) (v7)--(v24) (v7)--(v25)
(v8)--(v0) (v8)--(v3) (v8)--(v6) (v8)--(v7) (v8)--(v9) (v8)--(v10) (v8)--(v11) (v8)--(v12) (v8)--(v15) (v8)--(v16) (v8)--(v18) (v8)--(v19) (v8)--(v20) (v8)--(v21) (v8)--(v22) (v8)--(v23) (v8)--(v24) (v8)--(v25)
(v9)--(v2) (v9)--(v5) (v9)--(v7) (v9)--(v8) (v9)--(v11) (v9)--(v12) (v9)--(v13) (v9)--(v14) (v9)--(v16) (v9)--(v17) (v9)--(v18) (v9)--(v19) (v9)--(v20) (v9)--(v21) (v9)--(v22) (v9)--(v23) (v9)--(v24) (v9)--(v25)
(v10)--(v2) (v10)--(v5)(v10)--(v7) (v10)--(v8) (v10)--(v11) (v10)--(v12) (v10)--(v13) (v10)--(v14) (v10)--(v16) (v10)--(v17) (v10)--(v18) (v10)--(v19) (v10)--(v20) (v10)--(v21) (v10)--(v22) (v10)--(v23) (v10)--(v24) (v10)--(v25)
(v11)--(v1)(v11)--(v4) (v11)--(v6) (v11)--(v8) (v11)--(v9) (v11)--(v10) (v11)--(v13) (v11)--(v14) (v11)--(v15) (v11)--(v17) (v11)--(v18) (v11)--(v19) (v11)--(v20) (v11)--(v21) (v11)--(v22) (v11)--(v23) (v11)--(v24) (v11)--(v25)
(v12)--(v1) (v12)--(v4) (v12)--(v6) (v12)--(v8) (v12)--(v9) (v12)--(v10) (v12)--(v13) (v12)--(v14) (v12)--(v15) (v12)--(v17) (v12)--(v18) (v12)--(v19) (v12)--(v20) (v12)--(v21) (v12)--(v22) (v12)--(v23) (v12)--(v24) (v12)--(v25)
(v13)--(v0) (v13)--(v3) (v13)--(v6) (v13)--(v7) (v13)--(v9) (v13)--(v10) (v13)--(v11) (v13)--(v12) (v13)--(v15) (v13)--(v16) (v13)--(v18) (v13)--(v19) (v13)--(v20) (v13)--(v21) (v13)--(v22) (v13)--(v23) (v13)--(v24) (v13)--(v25)
(v14)--(v0) (v14)--(v3) (v14)--(v6) (v14)--(v7) (v14)--(v9) (v14)--(v10) (v14)--(v11) (v14)--(v12) (v14)--(v15) (v14)--(v16) (v14)--(v18) (v14)--(v19) (v14)--(v20) (v14)--(v21) (v14)--(v22) (v14)--(v23) (v14)--(v24) (v14)--(v25)
(v15)--(v2) (v15)--(v5) (v15)--(v7) (v15)--(v8) (v15)--(v11) (v15)--(v12) (v15)--(v13) (v15)--(v14) (v15)--(v16) (v15)--(v17) (v15)--(v18) (v15)--(v19) (v15)--(v20) (v15)--(v21) (v15)--(v22) (v15)--(v23) (v15)--(v24) (v15)--(v25)
(v16)--(v1) (v16)--(v4) (v16)--(v6) (v16)--(v8) (v16)--(v9) (v16)--(v10) (v16)--(v13) (v16)--(v14) (v16)--(v15) (v16)--(v17) (v16)--(v18) (v16)--(v19) (v16)--(v20) (v16)--(v21) (v16)--(v22) (v16)--(v23) (v16)--(v24) (v16)--(v25)
(v17)--(v0) (v17)--(v3) (v17)--(v6) (v17)--(v7) (v17)--(v9) (v17)--(v10) (v17)--(v11) (v17)--(v12) (v17)--(v15) (v17)--(v16) (v17)--(v18) (v17)--(v19) (v17)--(v20) (v17)--(v21) (v17)--(v22) (v17)--(v23) (v17)--(v24) (v17)--(v25)
(v18)--(v0) (v18)--(v1) (v18)--(v2) (v18)--(v3) (v18)--(v4) (v18)--(v5) (v18)--(v6) (v18)--(v7) (v18)--(v8) (v18)--(v9) (v18)--(v10) (v18)--(v11) (v18)--(v12) (v18)--(v13) (v18)--(v14) (v18)--(v15) (v18)--(v16) (v18)--(v17) (v18)--(v19) (v18)--(v20) (v18)--(v21) (v18)--(v22) (v18)--(v23) (v18)--(v24) (v18)--(v25)
(v19)--(v0) (v19)--(v1) (v19)--(v2) (v19)--(v3) (v19)--(v4) (v19)--(v5) (v19)--(v6) (v19)--(v7) (v19)--(v8) (v19)--(v9) (v19)--(v10) (v19)--(v11) (v19)--(v12) (v19)--(v13) (v19)--(v14) (v19)--(v15) (v19)--(v16) (v19)--(v17) (v19)--(v18) (v19)--(v20) (v19)--(v21) (v19)--(v22) (v19)--(v23) (v19)--(v24) (v19)--(v25)
(v20)--(v0) (v20)--(v1) (v20)--(v2) (v20)--(v3) (v20)--(v4) (v20)--(v5) (v20)--(v6) (v20)--(v7) (v20)--(v8) (v20)--(v9) (v20)--(v10) (v20)--(v11) (v20)--(v12) (v20)--(v13) (v20)--(v14) (v20)--(v15) (v20)--(v16) (v20)--(v17) (v20)--(v18) (v20)--(v19) (v20)--(v21) (v20)--(v22) (v20)--(v23) (v20)--(v24) (v20)--(v25)
(v21)--(v0) (v21)--(v1) (v21)--(v2) (v21)--(v3) (v21)--(v4) (v21)--(v5) (v21)--(v6) (v21)--(v7) (v21)--(v8) (v21)--(v9) (v21)--(v10) (v21)--(v11) (v21)--(v12) (v21)--(v13) (v21)--(v14) (v21)--(v15) (v21)--(v16) (v21)--(v17) (v21)--(v18) (v21)--(v19) (v21)--(v20) (v21)--(v22) (v21)--(v23) (v21)--(v24) (v21)--(v25)
(v22)--(v0) (v22)--(v1) (v22)--(v2) (v22)--(v3) (v22)--(v4) (v22)--(v5) (v22)--(v6) (v22)--(v7) (v22)--(v8) (v22)--(v9) (v22)--(v10) (v22)--(v11) (v22)--(v12) (v22)--(v13) (v22)--(v14) (v22)--(v15) (v22)--(v16) (v22)--(v17) (v22)--(v18) (v22)--(v19) (v22)--(v20) (v22)--(v21) (v22)--(v23) (v22)--(v24) (v22)--(v25)
(v23)--(v0) (v23)--(v1) (v23)--(v2) (v23)--(v3) (v23)--(v4) (v23)--(v5) (v23)--(v6) (v23)--(v7) (v23)--(v8) (v23)--(v9) (v23)--(v10) (v23)--(v11) (v23)--(v12) (v23)--(v13) (v23)--(v14) (v23)--(v15) (v23)--(v16) (v23)--(v17) (v23)--(v18) (v23)--(v19) (v23)--(v20) (v23)--(v21) (v23)--(v22) (v23)--(v24) (v23)--(v25)
(v24)--(v0) (v24)--(v1) (v24)--(v2) (v24)--(v3) (v24)--(v4) (v24)--(v5) (v24)--(v6) (v24)--(v7) (v24)--(v8) (v24)--(v9) (v24)--(v10) (v24)--(v11) (v24)--(v12) (v24)--(v13) (v24)--(v14) (v24)--(v15) (v24)--(v16) (v24)--(v17) (v24)--(v18) (v24)--(v19) (v24)--(v20) (v24)--(v21) (v24)--(v22) (v24)--(v23) (v24)--(v25)
(v25)--(v0) (v25)--(v1) (v25)--(v2) (v25)--(v3) (v25)--(v4) (v25)--(v5) (v25)--(v6) (v25)--(v7) (v25)--(v8) (v25)--(v9) (v25)--(v10) (v25)--(v11) (v25)--(v12) (v25)--(v13) (v25)--(v14) (v25)--(v15) (v25)--(v16) (v25)--(v17) (v25)--(v18) (v25)--(v19) (v25)--(v20) (v25)--(v21) (v25)--(v22) (v25)--(v23) (v25)--(v24);

\fill(v0)circle(2pt)coordinate(dot);\node[right] at (dot){$h$};
\fill(v1)circle(2pt)coordinate(dot);\node[right] at (dot){$x$};
\fill(v2)circle(2pt)coordinate(dot);\node[right] at (dot){$y$};
\fill(v3)circle(2pt)coordinate(dot);\node[right] at (dot){$2h$};
\fill(v4)circle(2pt)coordinate(dot);\node[right] at (dot){$2x$};
\fill(v5)circle(2pt)coordinate(dot);\node[right] at (dot){$2y$};
\fill(v6)circle(2pt)coordinate(dot);\node[above] at (dot){$h+x$};
\fill(v7)circle(2pt)coordinate(dot);\node[left] at (dot){$h+y$};
\fill(v8)circle(2pt)coordinate(dot);\node[left] at (dot){$x+y$};
\fill(v9)circle(2pt)coordinate(dot);\node[left] at (dot){$h+2x$};
\fill(v10)circle(2pt)coordinate(dot);\node[left] at (dot){$2x+h$};
\fill(v11)circle(2pt)coordinate(dot);\node[left] at (dot){$h+2y$};
\fill(v12)circle(2pt)coordinate(dot);\node[left] at (dot){$2h+y$};
\fill(v13)circle(2pt)coordinate(dot);\node[left] at (dot){$x+2y$};
\fill(v14)circle(2pt)coordinate(dot);\node[left] at (dot){$2x+y$};
\fill(v15)circle(2pt)coordinate(dot);\node[left] at (dot){$2h+2x$};
\fill(v16)circle(2pt)coordinate(dot);\node[left] at (dot){$2h+2y$};
\fill(v17)circle(2pt)coordinate(dot);\node[left] at (dot){$2x+2y$};
\fill(v18)circle(2pt)coordinate(dot);\node[left] at (dot){$h+x+y$};
\fill(v19)circle(2pt)coordinate(dot);\node[left] at (dot){$2h+x+y$};
\fill(v20)circle(2pt)coordinate(dot);\node[below] at (dot){$h+2x+y$};
\fill(v21)circle(2pt)coordinate(dot);\node[right] at (dot){$h+x+2y$};
\fill(v22)circle(2pt)coordinate(dot);\node[right] at (dot){$2h+2x+y$};
\fill(v23)circle(2pt)coordinate(dot);\node[right] at (dot){$2h+x+2y$};
\fill(v24)circle(2pt)coordinate(dot);\node[right] at (dot){$h+2x+2y$};
\fill(v25)circle(2pt)coordinate(dot);\node[right] at (dot){$2h+2x+2y$};
\end{tikzpicture}
\end{longtable}
\end{example}

\noindent \textbf{Remark.} Since $\operatorname{sol}(L)=\emptyset$, the vertex set is $L\backslash\{0\}$.

\begin{proposition}\label{invariants}
Let $L_{1}$ and $L_{2}$ be finite dimensional non-solvable Lie superalgebras.
If $L_{1}$ and $L_{2}$ are isomorphic as Lie superalgebras, then the following conclusions hold:
\begin{enumerate}[label=(\arabic*)]
\item The solvable graphs $\mathfrak{S}(L_{1})$ and $\mathfrak{S}(L_{2})$ are isomorphic as graphs.

\item The non-solvable graphs $\mathfrak{S}(L_{1})^{c}$ and $\mathfrak{S}(L_{2})^{c}$ are isomorphic as graphs.
\end{enumerate}
\end{proposition}
\begin{proof}
(1) Denote the vertex set of $\mathfrak{S}(L_{i})$ by $V_{i}$ and the edge set of $\mathfrak{S}(L_{i})$ by $E_{i}$, where $i=1,2$.
Suppose that $\varphi$ is an isomorphism from $L_{1}$ to $L_{2}$. By Theorem \ref{sol12345}, for any $x\in L_{1}$, we have

$$x \in \operatorname{sol}(L_{1}) \Leftrightarrow \varphi(x) \in \operatorname{sol}(L_{2}).$$
Therefore, there is a bijection from $V_{1}$ to $V_{2}$.

Then for any two distinct vertices $x, y \in V_1$,  by Definition \ref{graphsol} and Theorem \ref{sol12345}, it concludes that
$$
\begin{aligned}
\{x,y\} \in E_{1} &\Leftrightarrow \langle x, y \rangle \text{ is solvable} \\
&\Leftrightarrow \langle \varphi(x), \varphi(y) \rangle \text{ is solvable } \\
&\Leftrightarrow \{\varphi(x),\varphi(y)\} \in E_{2}.
\end{aligned}
$$
Hence, $\varphi$ induces a graph isomorphism from $\mathfrak{S}(L_1)$ to $\mathfrak{S}(L_2)$.

(2) Similarly with (1), $\varphi$ also induces a graph isomorphism from $\mathfrak{S}(L_1)^c$ to $\mathfrak{S}(L_2)^c$.
\end{proof}

\begin{corollary}
The solvability measure $\nu(L)$ is an isomorphism invariant, that is, if $L_1 \cong L_2$ then $\nu(L_1)=\nu(L_2)$.
\end{corollary}

\begin{proof}
By Theorem \ref{invariants}, the solvable graphs $\mathfrak{S}(L_1)$ and $\mathfrak{S}(L_2)$ are isomorphic.
Isomorphic graphs have the same number of vertices and edges, hence
$$\nu(L_{1})=1-2|E|/|V|(|V|-1)=\nu(L_{2}).$$
\end{proof}

\begin{proposition}\label{measure}
Let $L$ be a finite dimensional non-solvable Lie superalgebra with $\operatorname{char}\mathbb{F} \geq 3$.
Then the following properties hold:
\begin{enumerate}[label=(\arabic*)]
    \item $0 \leq \nu(L) \leq 1$.
    \item $\nu(L) = 0$ if and only if $\langle x,y \rangle$ is solvable for all $ x,y\in V$.
    \item $\nu(L) = 1$ if and only if $\langle x, y\rangle$ is non-solvable for every pair of distinct vertices $x,y\in V$.
\end{enumerate}
\end{proposition}
\begin{proof}
(1) Suppose that $\mathfrak{S}(L)= (V,E)$ is the solvable graph of $L$.
Since $E$ is a subset of the set of all unordered pairs from $V$, we have $0 \leq |E| \leq \binom{|V|}{2} = \frac{|V|(|V| - 1)}{2}$. Therefore,
$$0 \leq \nu(L)=1-\frac{2 |E|}{|V|(|V| - 1)} \leq 1.$$

(2) Assume that for every pair of distinct elements $x$ and $y$ in $V$, the subalgebra $\langle x, y\rangle$ is solvable,
then $$|E| = \binom{|V|}{2} = \frac{|V|(|V| - 1)}{2}.$$ Therefore,
$$\nu(L) = 1 - \frac{2 \cdot \frac{|V|(|V| - 1)}{2}}{|V|(|V| - 1)} = 0.$$

Conversely, the $\nu(L) = 0$ implies $|E| = \frac{|V|(|V| - 1)}{2}$.
This means that $E$ contains all edges from $\mathfrak{S}$.
Hence, for every pair of distinct elements $x$ and $y$ in $V$, the subalgebra $\langle x, y\rangle$ is solvable.

(3)$\nu(L) = 1$ if and only if $|E|=0$, which holds if and only if $\langle x,y \rangle$ is non-solvable for all $ x,y\in V$.
\end{proof}

\begin{definition}
Suppose that $L$ is a finite dimensional non-solvable Lie superalgebra and $x,y\in L$. Define the map $I_{L}: L\times L\rightarrow \mathbb{Z}_{2}$ by
\[
I_{L}(x, y) = \begin{cases}
1 & \langle x, y \rangle \text{ is solvable}, \\
0 & \text{otherwise}.
\end{cases}
\]
\end{definition}

\begin{lemma}\label{leqwer}
Suppose that $L_1$ and $L_2$ be finite dimensional non-solvable Lie superalgebras and $L=L_1\oplus L_2$. Then for any $u = (u_1, u_2), v = (v_1, v_2) \in L$,
we have
\begin{equation}\label{sol3.1}
I_{L}(u, v) = I_{L_1}(u_1, v_1) \cdot I_{L_2}(u_2, v_2).
\end{equation}
\end{lemma}
\begin{proof}
By definition, $\langle u, v \rangle$
is solvable if and only if both $\langle u_1, v_1 \rangle$ and $\langle u_2, v_2 \rangle$ are solvable.
Hence Equation \ref{sol3.1} is true.
\end{proof}

Let $L_1$ and $L_2$ be finite dimensional non-solvable Lie superalgebras.
Set $A_i = L_i \setminus (\operatorname{sol}(L_i)\cup \{0\})$ and $a_i = |A_i|$, $B_i = \operatorname{sol}(L_i)$ and $b_i = |B_i|$.
Define $\alpha_i = 1 - \nu(L_i)$, $\omega_{i}=\{ x \in A_i \mid \langle x \rangle \text{ is solvable} \}$
and $\nu_i = \dfrac{ |\omega_{i}| }{a_i}$, where $i=1,2.$ Next, with reference to \cite{ZT}, we can derive the solvablity measure formula for the
direct sum of Lie superalgebras.

\begin{proposition}
Let $L_1$ and $L_2$ be finite dimensional non-solvable Lie superalgebras with $\operatorname{char}\mathbb{F} \geq 3$.
Then the  solvability measure of $L = L_1 \oplus L_2$ is
\[
\nu(L) = 1-\frac{2Q}{|V|(|V|-1)},
\]
where
\begin{align*}
|V| &= a_1 a_2 + a_1 b_2 + b_1 a_2, \\
Q &=Q_{11}+Q_{22}+Q_{33}+Q_{12}+Q_{13}+Q_{23}\\
&=\frac{1}{2} a_1 a_2 \Bigl( \alpha_1 \alpha_2 (a_1-1)(a_2-1) + \nu_1 \alpha_2 (a_2-1) + \alpha_1 \nu_2 (a_1-1) \Bigr) \\
+& \frac{1}{2}a_1 b_2 \bigl( \alpha_1 b_2 (a_1-1) + \nu_1 (b_2-1) \bigr) \\
+& \frac{1}{2}a_2 b_1 \bigl( \alpha_2 b_1 (a_2-1) + \nu_2 (b_1-1) \bigr) \\
+& a_1 a_2 b_2 \bigl( \alpha_1 (a_1-1) + \nu_1 \bigr) \\
+& a_1 a_2 b_1 \bigl( \alpha_2 (a_2-1) + \nu_2 \bigr) \\
+& a_1 a_2 b_1 b_2.
\end{align*}
In particular, If $L_2$ has no solvable subalgebras, then
$\nu(L) = 1$.
\end{proposition}

Let $\mathfrak{S}(L)=(V_L, E_L)$ and $\mathfrak{S}(M)=(V_M, E_M)$ be two solvable graphs. A map $f_{\mathfrak{S}}: \mathfrak{S}(L) \to \mathfrak{S}(M)$ is called \emph{an admissible graph homomorphism} if:
\begin{enumerate}[label=(\arabic*)]
\item $\{x, y\}\in E_L$ if and only if $\{f_{\mathfrak{S}}(x), f_{\mathfrak{S}}(y)\}\in E_M$,
\item There is a surjective homomorphism $\varphi: L \to M$ of Lie superalgebras such that $f_{\mathfrak{S}}$ is the restriction of $\varphi$ to $V_L$.
\end{enumerate}

We define the solvable graph category $\mathcal{G}$, whose objects are all solvable graphs of finite dimensional Lie superalgebras and whose morphisms are all admissible graph homomorphisms.

\begin{proposition}
Define a map $\Gamma:\mathcal{L}\to\mathcal{G}$ by specifying its action:
\begin{itemize}
    \item \textbf{On objects}: For any finite dimensional Lie superalgebra $L\in\mathcal{L}$, define $\Gamma(L)=\mathfrak{S}(L)$;
    \item \textbf{On morphisms}: For any morphism $f:L\to M$ in $\mathcal{L}$, define $\Gamma(f)=f_{V_{L}}:\mathfrak{S}(L) \to \mathfrak{S}(M)$.
\end{itemize}

Then $\Gamma$ is a functor.
\end{proposition}
\begin{proof}
Obviously, $\Gamma$ preserves identity morphisms. Let $\phi:L\rightarrow M$ and $\varphi:M\rightarrow N$ be morphisms in$\mathcal{L}$. Then
$$\Gamma(\varphi\circ\phi)=(\varphi\circ\phi)_{V_{L}}=\varphi_{V_{M}}\circ\phi_{V_{L}}= \Gamma(\varphi)\circ \Gamma(\phi).$$
Thus $\Gamma$ is a well-defined functor.
\end{proof}

\begin{theorem}\label{let}
Let $L_1$ and $L_2$ be finite dimensional non-solvable Lie superalgebras with $\text{char}\ \mathbb{F} \geq 3$.
Let $\varphi: L_1 \to L_2$ be a surjective homomorphism such that $\ker\varphi \subseteq \operatorname{sol}(L_1)$. Then
\[
\nu(L_2) \geq \nu(L_1),
\]
with equality if and only if $\varphi$ is an isomorphism of Lie superalgebras.
\end{theorem}

\begin{proof}
By Theorem \ref{sol12345}, since $\ker\varphi \subseteq \operatorname{sol}(L_1)$ and $\varphi$ is surjective, we have
\[
\varphi(\operatorname{sol}(L_1)) = \operatorname{sol}(\varphi(L_1)) = \operatorname{sol}(L_2).
\]
By Definition \ref{graphsol}, $\varphi$ induces a surjective map
\[
\tilde{\varphi}: V(L_1) \to V(L_2).
\]
For any $u \in V(L_2)$, we have $|\varphi^{-1}(u)|=|\ker\varphi|$. Suppose that $|\ker\varphi|=k$, it follows that
\[
|V(L_1)| = k \cdot |V(L_2)|.
\]
By the proof of Theorem \ref{sol12345}, for any $x,y \in L_1$,
the subalgebra $\langle \varphi(x), \varphi(y) \rangle$ is solvable if and only if $\langle x,y \rangle$ is solvable.
We decompose the edge set $E(L_1)$ into two parts:
\begin{enumerate}[label=(\arabic*)]
    \item For distinct $u,v \in V(L_2)$, $\{u,v\} \in E(L_2)$ if and only if $\{x,y\} \in E(L_1)$
    for all $x \in \tilde{\varphi}^{-1}(u)$ and $y \in \tilde{\varphi}^{-1}(v)$.
    In this case, there is $k^2 |E_{L_2}|$ edges in $E_{L_1}$.
    \item For any $x,y \in \tilde{\varphi}^{-1}(u)$, we have $x-y \in \ker\varphi \subseteq sol(L_1)$.
    By the definition of the solvabilizer, $\langle x, x-y \rangle = \langle x,y \rangle$ is solvable.
    In this case, there is $|V(L_2)| \cdot \frac{k(k-1)}{2}$ edges in $E_{L_1}$.
\end{enumerate}
The total number of edges of \(\mathfrak{S}(L_1)\) is
\[
|E(L_1)| = k^2 |E(L_2)| + \frac{k(k-1)|V(L_2)|}{2}.
\]
Thus
\[
\begin{aligned}
\nu(L_1) &= 1 - \frac{2\left( k^2 |E(L_2)| + \frac{k(k-1)|V(L_2)|}{2} \right)}{k |V(L_2)| \cdot (k |V(L_2)| - 1)} \\
&= 1 - \frac{2k |E(L_2)| + (k-1)|V(L_2)|}{|V(L_2)| \cdot (k |V(L_2)| - 1)}\\
&= 1 - \frac{2k |E(L_2)|}{|V(L_2)| \cdot (k |V(L_2)| - 1)}- \frac{(k-1)|V(L_2)|}{|V(L_2)| \cdot (k |V(L_2)| - 1)}\\
&\leq 1 - \frac{2k|E(L_2)|}{|V(L_2)|(k|V(L_2)| - k)}\\
&= 1 - \frac{2|E(L_2)|}{|V(L_2)|(|V(L_2)| - 1)}\\
&= \nu(L_2).
\end{aligned}
\]
Obviously, $\nu(L_1)=\nu(L_2)$ if and only if $k=1$.
This completes the proof, with equality if and only if \(\varphi\) is an isomorphism.
\end{proof}

\begin{corollary}
For any $f_{\mathfrak{S}}: \mathfrak{S}(L) \to \mathfrak{S}(M)\in \mathcal{G}$, we have
$$\nu(M)\geq\nu(L),$$
with equality if and only if $f_{\mathfrak{S}}$ is an isomorphism in the category $\mathcal{G}$.
\end{corollary}
\begin{proof}
This follows directly from Theorem \ref{let}.
\end{proof}

\end{document}